\documentclass[12pt,reqno]{amsart}
\usepackage[a4paper]{geometry}
\usepackage{amsthm,hyperref,lmodern,mathrsfs,amssymb,amsmath,amsfonts}
\usepackage[latin1]{inputenc}
\usepackage{versions}

\usepackage{color}

\allowdisplaybreaks

\numberwithin{equation}{section}

\newtheorem{theo}{Theorem}[section]
\newtheorem{lemme}[theo]{Lemma}

\def\AA{{\textbf{(A)} \ \!\!}}

\def\cE{{\mathcal{E}\ \!\!}}
\def\E{{\mathbb{E}\ \!\!}}
\def\I{{\mathcal{I}\ \!\!}}

\def\Z{{\mathbb{Z}\ \!\!}}
\def\N{{\mathbb{N}\ \!\!}}
\def\R{{\mathbb{R}\ \!\!}}
\def\L{{\mathcal{L}\ \!\!}}
\def\D{{\mathcal{D}\ \!\!}}

\def\P{{\mathbb{P}\ \!\!}}

\def\C{{\mathcal{C}\ \!\!}}

\def\FF{{\mathcal{F}\ \!\!}}
\def\IC{{\mathcal{I} \mathcal{C} \ \!\!}}
\def\DC{{\mathcal{D} \mathcal{C} \ \!\!}}

\def\Var{{\mathrm{{\rm Var}}}}

\def\Ent{{\mathrm{{\rm Ent}}}}

\def\Sign{{\mathrm{{\rm Sign}}}}

\def\Lip{{\mathrm{{\rm Lip}}}}
\def\and{{\mathrm{{\rm and}}}}

\def\CP{{c_\mathrm{{\rm P}}}}
\def\CLS{{c_\mathrm{{\rm LS}}}}
\def\CB{{c_{\mathrm{{\rm B}}_p}}}

\def\vphi{{\varphi\ \!\!}}

\begin{document}

\title[Intertwining relations for diffusions]{Intertwining relations for one-dimensional diffusions and application to functional inequalities}

\author{Michel~Bonnefont} %
\address[M.~Bonnefont]{UMR CNRS 5251, Institut de Math\'ematiques de Bordeaux, Universit\'e Bordeaux 1, France} %
\email{\url{mailto:michel.bonnefont(at)math.u-bordeaux1.fr}} %
\urladdr{\url{http://www.math.u-bordeaux1.fr/~mibonnef/}}

\author{Ald\'eric~Joulin} %
\address[A.~Joulin, corresponding author]{UMR CNRS 5219, Institut de Math\'ematiques de Toulouse, Universit\'e de Toulouse, France}%
\email{\url{mailto:ajoulin(at)insa-toulouse.fr}}%
\urladdr{\url{http://www-gmm.insa-toulouse.fr/~ajoulin/}}



\maketitle

%------------------------------------------------

\begin{abstract}
Following the recent work \cite{cha_jou} fulfilled in the discrete case, we provide in this paper new intertwining relations for semigroups of one-dimensional diffusions. Various applications of these results are investigated, among them the famous variational formula of the spectral gap derived by Chen and Wang \cite{chen_wang2} together with a new criterion ensuring that the logarithmic Sobolev inequality holds. We complete this work by revisiting some classical examples, for which new estimates on the optimal constants are derived.
\end{abstract}

\normalsize

\vspace{0.5cm}

\section{Introduction}
\label{sect:intro} \textrm{} \indent It is by now well known that commutation relations and convexity are of great importance in the analysis and geometry of Markov diffusion semigroups \cite{bakry, ledoux}. Keeping in mind the application to functional inequalities such as Poincar\'e or logarithmic-Sobolev inequalities, various tools have been developed by several authors to obtain this type of commutation results and among them the famous $\Gamma_2$ calculus introduced by Bakry and \'Emery in \cite{bakry_emery}. We refer to the set of notes of Ledoux \cite{ledoux2} and also to the forthcoming book \cite{bgl} for a clear and pleasant introduction on the topic, with historical references and credit. In view to provide new functional inequalities on discrete spaces, Chafa\"i and the second author recently investigated in \cite{cha_jou} the case of birth-death processes, which are the discrete space analogue of diffusion processes. The key point in the underlying analysis relies on simple intertwining relations involving a family of discrete gradients and two different generators: the first one is that of the birth-death process considered, whereas the second one is a Schr\"odinger generator associated to a dual process.

Coming back to the diffusion setting, a natural question arises: are these intertwining techniques tractable and convenient to address the geometry of diffusion semigroups, giving rise to new results in this large body of work ? The purpose of these notes is to convince the reader of the relevance of this approach and to investigate several consequences of the intertwining approach, at least in the one-dimensional case. Among the potential applications provided by this alternative point of view, we recover on the one hand the famous variational formula of Chen and Wang \cite{chen_wang2} for the spectral gap, and on the other hand we are able to give a new condition ensuring that the logarithmic Sobolev inequality is satisfied with a computable constant. In particular, this criterion turns out to be efficient in the situation when the Bakry-\'Emery criterion fails, as for instance in the case of a potential which is not strictly convex on the real line.

The paper is organized as follows. In section 2, we recall some basic material on Markov diffusion processes. Then the desired intertwining relations are derived in Section 3 with our main results Theorems~\ref{theo:intert_a} and \ref{theo:bicommutation}. Section 4 is devoted to applications of the previous results to functional inequalities and, in the final section, we revisit classical examples for which we provide new estimates on the optimal constants.

\section{Preliminaries}
\label{sect:preli} \textrm{} \indent

Let $\C ^\infty (\R)$ be the space of infinitely differentiable real-valued functions on $\R$ and let $\C _0 ^\infty (\R)$ be the subspace of $\C ^\infty (\R)$ consisting of compactly supported functions. Denote $\Vert \cdot \Vert _\infty$ the essential supremum norm with respect to the Lebesgue measure. We mention that in the present paper we use the terminology \textit{increasing} and \textit{decreasing} for the expressions \textit{non-decreasing} and \textit{non-increasing}, respectively. The main protagonist of the present paper is a Sturm-Liouville operator on $\R$, which is the second-order diffusion operator defined on $\C ^\infty (\R)$ by
$$
\L f := \sigma ^2 \, f'' + b \, f' ,
$$
where $b$ and $\sigma$ are two real-valued measurable functions on $\R$. Let $\Gamma$ be the \textit{carr\'e du champ} operator which is the bilinear symmetric form defined on $\C ^\infty (\R) \times \C ^\infty (\R)$ by
\[
\Gamma (f,g) := \frac{1}{2}\,\left(\L (fg)-f \, \L g -g \, \L f \right) \, = \, \sigma ^2 \, f' \, g' .
\]

Let us introduce some set of hypothesis on the diffusion operator $\L$ that we will refer to as assumption $\AA$ in the sequel: \vspace{0.1cm}

$\circ$ Smoothness: the functions $b$ and $\sigma$ belong to $\C ^\infty (\R)$. \vspace{0.1cm}

$\circ$ Ellipticity: it means that the diffusion function $\sigma$ is non-degenerate, i.e. $\sigma(x) >0$ for every $x\in \R$. \vspace{0.1cm}

$\circ$ Completeness: the metric space $(\R, d_\sigma )$ is complete, where $d_\sigma $ is the distance associated to $\L$, i.e.
$$
d_\sigma (x,y) := \sup \, \left \{ \vert f(x) - f(y) \vert : f\in \C ^\infty (\R) , \Vert \Gamma (f,f) \Vert _\infty \leq 1 \right\} , \quad x,y\in \R ,
$$
which can be rewritten as
$$
d_\sigma (x,y) = \left \vert \int_x ^y \frac{du}{\sigma(u) } \right \vert .
$$

Let $U$ be the potential defined by
\begin{equation}
\label{eq:U}
U(x) := U(0) - \int _{0} ^x \frac{b(u)}{\sigma (u) ^2} \, du , \quad x\in \R ,
\end{equation}
where $U(0)$ is some arbitrary constant and let $\mu$ be the measure with Radon-Nikodym derivative $h := e^{-U} /\sigma^2$ with respect to the Lebesgue measure. Simple computations show that for every $f\in\mathcal{C}_0 ^\infty(\R)$,
$$
\int_\R \L f \, d\mu  = 0 ,
$$
and moreover the operator $\L$ is symmetric on $\C_0 ^\infty(\R)$ with respect to the measure $\mu$, i.e. for every $f,g \in\mathcal{C}_0 ^\infty(\R)$,
\[
\cE _\mu (f,g) := \int_\R \Gamma (f,g) \, d\mu = - \int_\R f \L g \, d\mu = - \int_\R \L f \, g \, d\mu  = \int_\R \sigma ^2 \, f' \, g' \, d\mu ,
\]
hence $\L$ is non-positive on $\C_0 ^\infty(\R)$. Under $\AA$, the operator $(\L , \C_0 ^\infty(\R))$ is essentially self-adjoint in $L^2 (\mu)$, that is, it admits a unique self-adjoint extension (still denoted $\L$) with domain $\D (\L) \subset L^2(\mu)$ in which the space $\C_0 ^\infty(\R)$ is dense for the norm
$$
\Vert f\Vert _{\L} := \left( \Vert f \Vert ^2 _{L^2 (\mu)} + \Vert \L f \Vert ^2 _{L^2 (\mu)} \right) ^{1/2} .
$$
In other words the space $\C_0 ^\infty(\R)$ is a core of the domain $\D (\L)$. By spectral theorem, the operator $\L$ generates a unique strongly continuous symmetric semigroup $(P_t)_{t\geq 0}$ on $L^2(\mu)$ such that for every function $f \in L^2 (\mu)$ and every $t>0$, we have $\partial _t P_t f = \L P_t f $. Here, the notation $\partial_u$ stands for the derivative with respect to some parameter $u$, and will be used all along the paper. The semigroup preserves positivity (it transforms positive functions into positive functions) but it is only sub-Markov \textit{a priori} in the sense that $P_t 1$ is less than 1 since it may happen that $1 \notin \D (\L)$. Recall that $\mu$ may have infinite mass and thus $P_t 1$ has to be defined as the increasing limit of $P_t f_n$ as $n\to +\infty$, where $(f_n)_{n\in \N}$ is a sequence of positive functions belonging to $\C_0 ^\infty(\R)$ and increasing pointwise to the constant function 1. \\
The closure $(\cE _\mu , \D (\cE _\mu))$ of the bilinear form $(\cE _\mu , \C_0 ^\infty(\R))$ is a Dirichlet form on $L^2(\mu)$ and by spectral theorem we have the dense inclusion $\D (\L) \subset \D (\cE _\mu)$ for the norm
$$
\Vert f\Vert _{\cE _\mu} := \left( \Vert f \Vert ^2 _{L^2 (\mu)} + \cE _\mu (f,f) \right) ^{1/2} .
$$

As probabilists, we are interested by a diffusion process having the operator $\L$ as infinitesimal generator. Under appropriate growth conditions on the functions $\sigma$ and $b$, such a process corresponds to the unique solution (up to the explosion time) of the following Stochastic Differential Equation (in short SDE),
$$
dX_t = \sqrt{2} \, \sigma(X_t) \, dB_t + b(X_t) \, dt,
$$
where $B:= (B_t)_{t\geq 0}$ is a standard Brownian motion on a given filtered probability space $(\Omega, \FF , (\FF_t)_{t\geq 0}, \P)$. If $\zeta \in (0 , +\infty ]$ denotes the explosion time defined as the almost sure limit of the sequence of stopping times
$$
\tau_n := \inf \{ t\geq 0 : \vert X_t \vert \geq n \} , \quad n \in \N ,
$$
then on the event $\zeta < +\infty$ we have $\lim_{t\to \zeta} X_t  = + \infty$ or $-\infty$ almost surely. Denote $\Delta$ this limit and define $X_t$ to be $\Delta$ when $t \geq \zeta$. Then we obtain a process which takes its values in the one-point compactification space $\R \cup \{ \Delta\}$ equipped with the natural one-point compactification topology. Define the space $\C ^\infty _0 (\R ^\Delta )$ as the extension to $\R \cup \Delta$ of the space $\C ^\infty _0 (\R)$, i.e. every function $f \in \C ^\infty _0 (\R)$ can be extended to a continuous function on $\R \cup \{ \Delta\}$ by letting $f(\Delta) = 0$. In terms of semigroup, we have for every $f \in \C ^\infty _0 (\R ^\Delta)$,
$$
P_t f (x) = \E _x [f(X_{t \wedge \zeta})] = \E _x [f(X_{t}) \, 1 _{\{ t < \zeta \}} ],
$$
where $\E_x$ stands for the conditional expectation knowing the initial state $x$. A necessary and sufficient condition in terms of the functions $b$ and $\sigma$ ensuring the non-explosion of the process is the following, cf. \cite{ikeda}:
\[
\int _0 ^{+\infty} e^{U(x)} \int _0 ^x \frac{e^{- U(y)}}{\sigma (y) ^2} \, dy \, dx = +\infty , \quad \int _{-\infty} ^0 e^{U(x)} \int _x ^0 \frac{e^{- U(y)}}{\sigma (y) ^2} \, dy \, dx = +\infty .
\]

When the measure $\mu$ is finite, i.e. the function $h$ is integrable with respect to the Lebesgue measure, then the process is positive recurrent. In this case and up to renormalization one can assume in the sequel that $\mu$ is a probability measure, the normalizing constant being hidden in the very definition of $U(0)$. The symmetry property on the generator means that the probability measure $\mu$ is time-reversible with respect to these dynamics. \vspace{0.1cm}

To obtain an intertwining relation between gradient and semigroup, we recall the method of the tangent process. Assume $\AA$ and that the function
$$
V_\sigma := \frac{\L \sigma }{\sigma} - b' = \sigma \sigma '' + b \, \frac{\sigma '}{\sigma} - b' ,
$$
is bounded from below by some real constant $\rho_\sigma$. Then the process does not explode in finite time and is moreover positive recurrent if $\rho_\sigma >0$, cf. for instance \cite{bakry_explo}. Denoting $X^x$ the process starting from $x$, the application $x\mapsto X_t ^x $ is almost surely of class $\C ^1$ and we have the linear SDE
$$
\partial_x  X_t ^x = 1 + \sqrt{2} \, \int_0 ^t \sigma '(X_s ^x ) \, \partial_x  X_s ^x \, dB_s + \int_0 ^t b'(X_s ^x) \, \partial_x  X_s ^x \, ds.
$$
Hence by It\^o's formula the solution of this SDE, usually called the tangent process, admits the representation
\[
\partial_x  X_t ^x = \frac{\sigma(X_t ^x)}{\sigma(x)} \, \exp \left( - \int _0 ^t V_\sigma (X_s ^x) \, ds \right) .
\]
Finally differentiating the semigroup and using the chain rule entail the following intertwining relation, available for every $f\in \C _0 ^\infty (\R)$,
\[
(P_t f)'(x) = \E \left[ f'(X_t ^x )\, \frac{\sigma(X_t ^x)}{\sigma(x)} \, \exp \left( - \int _0 ^t V_\sigma (X_s ^x) \, ds \right) \right] .
\]
If $\nabla _\sigma$ denotes the weighted gradient $\nabla_\sigma  f = \sigma \, f'$ then the latter identity rewrites in terms of semigroup as
\begin{equation}
\label{eq:intert}
\nabla_\sigma P_t f = P_t ^{V_\sigma} \nabla_\sigma f,
\end{equation}
with $(P_t ^{V_\sigma})_{t\geq 0}$ the Feynman-Kac semigroup with potential $V_\sigma$. In particular using Cauchy-Schwarz' inequality, we recover the well-known sub-commutation inequality \cite{bakry_emery, capitaine},
\begin{equation}
\label{eq:subcommut}
\Gamma (P_t f,P_t f) \leq e^{-2\rho_\sigma t} \, P_t \Gamma(f,f) ,
\end{equation}
usually obtained through the Bakry-\'Emery criterion involving the $\Gamma_2$ calculus (we will come back to this point later). The basic example we have in mind is the Langevin diffusion with $\sigma = 1$ and for which the potential $U$ defined in \eqref{eq:U} satisfies $U' = -b$. Under the assumption that the second derivative of $U$ exists and is lower bounded, \eqref{eq:intert} yields
$$
(P_t f)'(x) = \E \left[ f'(X_t ^x ) \, \exp \left( - \int _0 ^t U''(X_s ^x) \, ds \right) \right] ,
$$
a formula appearing in \cite{malrieu_talay}. In particular, the convexity of the potential $U$ plays a key role to obtain commutation relations of type \eqref{eq:subcommut}, at the heart of the famous Bakry-\'Emery theory. \vspace{0.1cm}

To conclude this part, let us briefly observe that the Feynman-Kac semigroup admits a nice interpretation in terms of killing when the potential $V_\sigma$ is non-negative. Let $\epsilon$ be an exponential random variable independent of the process $X$ and set
\[
\xi := \inf \{ t \geq 0 : \int_0 ^t V_\sigma (X_s) \, ds > \epsilon \} .
\]
Define on the space $\R \cup \{ \Delta\}$ the process $\tilde{X}$ by
\[
\tilde{X}_t = \left \{
\begin{array}{lll}
X_t & \mbox{ if } & t < \xi \\
\Delta & \mbox{ if } & t \geq \xi
\end{array}
\right.
\]
which is nothing but the process $X$ killed at time $\xi$. If $(\tilde{P}_t)_{t\geq 0}$ stands for the associated semigroup then for every $f\in \C ^\infty _0 (\R ^\Delta)$ we have $\tilde{P} _t f = P_t ^{V_\sigma} f$ so that the intertwining relation \eqref{eq:intert} rewrites as
\[
\nabla _\sigma P_t f  = \tilde{P} _t \nabla_\sigma f .
\]

\section{Intertwining relations}
\label{sect:main}
Let $\C _+ ^\infty (\R) $ be the subset of $\C ^\infty (\R) $ consisting of positive functions. Let us fix some $a\in \C _+ ^\infty (\R)$. We introduce a new Sturm-Liouville operator $\L _a$ defined on $\C ^\infty (\R)$ by
$$
\L _a f = \sigma^2 \, f'' + b_a \, f' , \quad \mbox{ with } \quad b_a := 2\sigma \sigma ' +b - 2\sigma ^2 \, \frac{a'}{a}.
$$
Since the diffusion function $\sigma$ is the same as for the first dynamics, assumption $\AA$ is satisfied for $\L_a$ as soon as it is for $\L$. Note that the drift $b_a$ may be rewritten as
$$
b_a = b +2\sigma^2 \, \frac{h'}{h} , \quad \mbox{ with } \quad h := \frac{\sigma}{a}.
$$
Therefore under $\AA$ the operator $(\L_a, \C ^\infty _0 (\R))$ is essentially self-adjoint in $L^2(\mu_a)$, where the measure $\mu_a$ is given by $d\mu_a := (\sigma /a)^2 \, d\mu$. Denote $(\L _a, \D (\L_a))$ the unique self-adjoint extension and let $(P_{a,t})_{t\geq 0}$ be the associated strongly continuous symmetric semigroup. Denote $(\cE _{\mu _a}, \D (\cE _{\mu _a}))$ the Dirichlet form on $L^2(\mu_a)$ corresponding to the closure of the pre-Dirichlet form $(\cE _{\mu_a} , \C_0 ^\infty(\R))$. Let $X^a$ be the underlying process solution to the SDE
$$
dX^a _t = \sqrt{2} \, \sigma(X^a _t) \, dB_t + b_a (X^a _t) \, dt ,
$$
up to the possible explosion time. This explosion time will be almost surely infinite if and only if we have
\[ \int _0 ^{+\infty} \left(\frac{a(x)}{\sigma(x)}\right) ^2 \, e^{U(x)} \int _0 ^x \frac{e^{- U(y)}}{a(y) ^2} \, dy \, dx = +\infty , \quad \int _{-\infty} ^0 \left(\frac{a(x)}{\sigma(x)}\right) ^2 \, e^{U(x)} \int _x ^0 \frac{e^{- U(y)}}{a(y) ^2} \, dy \, dx = +\infty .
\]
In particular if we have $a\asymp \sigma$, i.e. there exist two positive constant $m,M$ such that
\[
m \, \sigma \leq a \leq M \, \sigma,
\]
then the processes $X$ and $X^a$ are of the same nature (both explosive or not, both positive recurrent or not). As we will see below with the discussion involving the notion of $h$-transform, the processes $X$ and $X^a$ can be seen as a dual processes. Note also that if $a = \sigma$ then $\L$ and $\L_a$ coincide and thus we will write $\L$, $P_t$ and $\mu$ for $\L _\sigma$, $P_{\sigma,t}$ and $\mu_\sigma$, respectively. \vspace{0.1cm}

For a given function $a\in \C _+ ^\infty (\R)$, define the function
\begin{equation}
\label{eq:Va}
V_a := \frac{\L _a (a)}{a} - b' ,
\end{equation}
and assume that $V_a$ is bounded from below. Then the Schr\"odinger operator $\L _a ^{V_a} := \L_a - V_a$ is essentially self-adjoint on $\C ^\infty _0 (\R)$ in $L^2(\mu_a)$, cf. for instance \cite{oleinik}. In particular the following unicity result holds. Below $(P_{a,t}^{V_a})_{t\geq 0}$ denotes the associated Feynman-Kac semigroup.
\begin{lemme}
\label{lemme:schro}
Assume $\AA$ and that $V_a$ is bounded from below. Then for every $g\in L^2 (\mu_a)$, the Schr\"odinger equation
\[
\left \{
\begin{array}{lll}
\partial_t u & = & \L_a ^{V_a} u \\
u(\cdot ,0) & = & g
\end{array}
\right.
\]
admits a unique solution in $L^2 (\mu_a)$ given by $u (\cdot , t) = P_{a,t} ^{V_a} g$.
\end{lemme}

For $a\in \C _+ ^\infty (\R)$, denote $\nabla _a$ the weighted gradient $\nabla _a f = a \, f'$. Now we are in position to state an intertwining relation involving the gradient $\nabla_a$ and the semigroups $(P_t)_{t\geq 0}$ and $(P_{a,t} ^{V_a})_{t\geq 0}$.
\begin{theo}
\label{theo:intert_a} Assume $\AA$ and that $V_a$ is bounded from below for some function $a \in \C _+ ^\infty (\R)$. Letting $f\in \D (\cE _\mu )$, then the following intertwining relation holds:
\begin{equation}
\label{eq:intert_a}
\nabla _a P_t f = P_{a,t} ^{V_a} \nabla _a f , \quad t \geq 0.
\end{equation}
\end{theo}
\begin{proof}
The key point of the proof is the following intertwining relation at the level of the generators, which can be performed by simple computations:
\begin{equation}
\label{eq:commut_a}
\nabla_a \L f = \L_a ^{V_a} \nabla_a f .
\end{equation}
In order to extend such a property to the semigroups, we will use the uniqueness property of Lemma~\ref{lemme:schro}. Define the function $J$ on $\R_+$ by
\[
J(t) :=  \nabla_a P_{t} f .
\]
First we have $J(t) \in L^2(\mu_a)$ for every $t \geq 0$. Indeed,
\begin{eqnarray*}
\int_\R J(t) ^2 \, d\mu_a & = & \int_\R  \left(\nabla_a P_{t} f\right)^2 \, d\mu_a \\
& = & \int_\R  \left(\nabla_\sigma P_{t} f\right)^2 \, d\mu\\
& = & - \int_\R P_t f \, \L P_t f \, d\mu  ,
\end{eqnarray*}
where we used integration by parts (recall that all the elements involved above belong to the domain $\D (\L)$). Let $K$ be the function
$$
K(t) := - \int_\R P_t f \, \L P_t f \, d\mu  .
$$
Differentiating with respect to the time parameter and using integration by parts yield
\[
\partial _t K(t) = - 2 \, \int_\R  \left( \L P_t f \right)^2 \, d\mu \leq 0,
\]
so that $K$ is decreasing. Finally we obtain
\[
\int_\R J(t) ^2 \, d\mu_a \leq - \int_\R f \, \L f \, d\mu  = \cE _\mu \left( f, f \right) ,
\]
which is finite since $f\in \D (\cE _\mu)$. Moreover we have $\lim_{t\to 0} J(t) = \nabla_a f$ where the limit is taken in $L^2(\mu_a)$. Hence by the intertwining relation \eqref{eq:commut_a}, we have in $L^2(\mu_a)$,
\[
\partial_t J(t) =  \nabla_a \L P_{t} f  =  \L_a ^{V_a} \nabla_a P_{t} f = \L_a ^{V_a} J(t).
\]
Therefore by Lemma~\ref{lemme:schro} the function $J$ is the unique solution to the Schr\"odinger equation associated to $\L_a$, with Feynman-Kac potential $V_a$ and initial condition $g = \nabla _a f$. We thus conclude that
\[
J(t)= P_{a,t} ^{V_a} \nabla_a f , \quad t\geq 0.
\]
The proof is now complete.
\end{proof}

Note that the Feynman-Kac potential $V_a$ is computed to be
\begin{eqnarray*}
V_a & = & \frac{\L _a (a)}{a} - b' \\
& = & \sigma^2 \, \frac{a''}{a}+ (b+ 2\sigma \sigma') \, \frac{a'}{a} - 2\sigma^2 \, \left(\frac{a'}{a}\right)^2 -b' .
\end{eqnarray*}
If $g$ is a smooth  function with $g' >0$, consider the weight $a :=1/g ' $. Then the intertwining \eqref{eq:commut_a} at the level of the generators entails the following simple expression for $V_a$:
\[
V_a = - \nabla_a \L g = - \frac{ (\L g)'}{g'} .
\]

Similarly to the distance $d_\sigma $ introduced above, let us define a new distance $d_a$ on $\R$ as follows:
$$
d_a (x,y) := \left \vert \int _{x} ^{y} \frac{du}{a(u)} \right \vert , \quad x, y \in \R .
$$
The space of Lipschitz functions with respect to this metric is denoted $\Lip (d_a)$, hence a function $f$ lies in $\Lip (d_a)$ if and only if the associated Lipschitz seminorm is finite:
$$
\Vert f \Vert _{\Lip (d_a)} := \sup _{x\neq y} \frac{\vert f(x) - f(y) \vert }{d_a (x,y)} .
$$
By Rademacher's theorem a function $f$ is $\kappa$-Lipschitz in the previous sense if and only if $f$ is differentiable almost everywhere and $\Vert \nabla_a f \Vert _\infty \leq \kappa$. Hence from Theorem~\ref{theo:intert_a} we deduce that the space $\Lip (d_a) \cap \D (\cE _\mu)$ is stable by the semigroup $(P_t)_{t\geq 0}$ and moreover for every $t \geq 0$,
\[
\Vert \nabla_a P_t f \Vert _{\infty} \leq e^{-  \rho_a t} \, \Vert \nabla_a f \Vert _{\infty} ,
\]
where $\rho_a \in \R $ is a lower bound on the Feynman-Kac potential $V_a$. In particular if $\rho _a >0$ then the process is positive recurrent and the convergence to equilibrium holds exponentially fast in Wasserstein distance, cf. \cite{chen_li}.
\vspace{0.1cm}

Another remark is the following. If we choose $a=\sigma$ in \eqref{eq:intert_a} then our result fits with the classical Bakry-\'Emery theory. Let $\Gamma_2$ be the bilinear symmetric form defined on $\C ^\infty (\R) \times \C ^\infty (\R)$ by
\[
\Gamma _2 (f,g) = \frac{1}{2} \, \left( \L \Gamma (f,g) - \Gamma (f, \L g) - \Gamma (g, \L f) \right) .
\]
Then the Bakry-\'Emery criterion, which ensures the sub-commutation relation \eqref{eq:subcommut}, reads as follows: there exists some constant $\rho_\sigma \in \R$ such that for every $f\in \C ^\infty (\R)$,
\begin{equation}
\label{eq:BE}
\Gamma_2 (f,f) \geq \rho_\sigma \, \Gamma (f,f) .
\end{equation}
By simple computations one obtains
\[
\Gamma_2 (f,f) = \sigma^2 \, \left( \sigma f'' + \sigma' f \right) ^2 + V_\sigma  \, \Gamma(f,f).
\]
Therefore the best lower bound $\rho_\sigma$ leading to \eqref{eq:BE} is given by $\inf_{x\in \R} V_\sigma(x)$. \vspace{0.1cm}

As announced, the processes $X^a$ and $X$ can be interpreted as dual processes according to the so-called Doob's $h$-transform that we introduce now. Given a smooth positive function $h$, Doob's $h$-transform of the Feynman-Kac semigroup $(P_t^V)_{t\geq 0}$ with smooth potential $V$ consists in modifying it by ``multiplying inside and dividing outside" by the function $h$. In other words, we consider the new semigroup
$$
{P_t^V} ^{(h)} f := \frac{P_t^V (hf)}{h} .
$$
If $\L ^V := \L - V$ stands for the Schr\"odinger operator associated to $(P_t^V)_{t\geq 0}$, then the generator of $({P_t^V} ^{(h)})_{t\geq 0}$ is thus given by ${\L^V} ^{(h)} f = \L^V (hf)/h $, which rewrites by the chain rule formula as the following Schr\"odinger operator with Feynman-Kac potential $\L h/h -V $:
\begin{eqnarray*}
{\L^V} ^{(h)} f & = & \L f + 2 \, \frac{\Gamma (h,f)}{h} + \left(\frac{\L h}{h}-V\right)  \, f \\
& = & \sigma ^2 \, f'' + \left( b + 2\sigma ^2 \, \frac{h'}{h} \right) \, f' + \left(\frac{\L h}{h}-V\right)  \, f.
\end{eqnarray*}
When $V=0$, it is known that Doob's $h$-transform is Markov if and only if $h$ is $\L$-harmonic, i.e. $\L h = 0$. Moreover it exhibits the following group structure: if $h$ and $k$ are two smooth positive functions, then the $hk$-transform is nothing but the $h$-transform of the $k$-transform. In particular the $h$-transform and the original dynamics have the same distribution if and only if $h$ is constant. \vspace{0.1cm}

Returning to the process $X^a$, recall that the drift $b_a$ is given by
$$
b_a = b + 2 \sigma ^2 \, \frac{h'}{h}, \quad \mbox{ with } \quad h:= \sigma/a .
$$
Therefore this leads to the following interpretation of the intertwining relation \eqref{eq:intert_a}:  \vspace{0.1cm}

1 - first perform the classical intertwining \eqref{eq:intert} by using the method of the tangent process, so that we obtain a Feynman-Kac semigroup $(P_t^V g)_{t\geq 0}$ with $g = \nabla_\sigma f$ and potential $V = V_\sigma = \L \sigma /\sigma -b'$; \vspace{0.1cm}

2 - then apply Doob's $h$-transform with $h:= \sigma /a$ to obtain the desired result. In particular, the Feynman-Kac potential $V_a$ appears in \eqref{eq:intert_a} because the following identity holds:
\begin{equation}\label{eq:h-transf}
{\L^{V_\sigma}} ^{(\sigma /a)} f = \L _a ^{V_a} f,
\end{equation}
and in particular at the level of the Feynman-Kac potentials,
\[
\frac{\L _a (a)}{a} - b' \quad = \underbrace{ \frac{\L \sigma}{\sigma} - b'} _{\mbox{classical intertwining}} - \underbrace{\frac{\L  (\sigma/a)}{\sigma /a} . }_{\mbox{Doob's $\sigma /a$ -transform}}
\]

Let us say some words about the potential extension of the intertwining approach. As we will see below, it is possible to adapt the proof of Theorem~\ref{theo:intert_a} for diffusions on a compact interval $[\alpha, \beta]$. Under assumption $\AA$ restricted to $[\alpha, \beta]$ (in particular the completeness hypothesis is removed and the smooth functions $b$ and $\sigma$ are also assumed to be smooth at the boundaries $\alpha$ and $\beta$), the operator $(\L , \C ^\infty _0 (\R))$ is no longer essentially self-adjoint in $L^2 (\mu)$, where $\mu$ is the restriction to $[\alpha, \beta]$ of the original measure defined on $\R$, thus it admits different self-adjoint extensions. To overcome this problem, one needs to impose boundary conditions. For instance the choice of Neumann boundary conditions allows the semigroup to be Markov, in contrast to the one involved for instance with Dirichlet boundary conditions which is only sub-Markov, i.e. the explosion time is finite almost surely. In particular the Neumann semigroup is stable on functions with 0 derivative at the boundary whereas the Dirichlet semigroup is stable on functions vanishing at the boundary, similarly to the one corresponding to the one-point compactification space $\R \cup \{ \Delta \}$ emphasized above. At the level of the process, the Neumann diffusion corresponds to the process reflected at the boundary whereas the Dirichlet diffusion is the process killed at the boundary. \vspace{0.1cm}

To illustrate the discussion, let us consider the basic example of Brownian motion (with speed 2) on the interval $[0,1]$. If $(P_t)_{t\geq 0}$ and $(\tilde{P_t})_{t\geq 0}$ stand for the Neumann and Dirichlet semigroups respectively, then both solve the heat equation
\[
\left \{
\begin{array}{lll}
\partial_t u & = & \partial _x ^2 u \\
u(\cdot ,0) & = & f ,
\end{array}
\right.
\]
for every smooth enough function $f$. Then boundary conditions allow us to identify the underlying kernels. If $q$ stands for the heat kernel on $\R$, that is,
\[
q_t (x,y) = \frac{1}{\sqrt{4 \pi t }} \, e^{- \frac{(x-y)^2}{4t}} , \quad x,y \in \R ,
\]
then the kernels on $[0,1]$ can be constructed with respect to the kernel $q$: for every $x,y\in [0,1]$ we have for the Neumann semigroup
$$
p_t (x,y) = \sum_{k\in \Z} \left( q_t (x,y+2k) + q_t (x,-y+2k) \right),
$$
whereas for the Dirichlet semigroup,
$$
\tilde{p_t} (x,y) = \sum_{k\in \Z} \left( q_t (x,y+2k) - q_t (x,- y+2k) \right) .
$$
Then simple computations yield to the intertwining relation between the two semigroups:
\[
( P_t f ) ' = \tilde{P_t} (f') ,
\]
which leads by Jensen's inequality to the sub-commutation
\[
\left \vert ( P_t f ) ' \right \vert \leq \tilde{P_t} (\left \vert f' \right \vert ) \leq P_t (\left \vert f' \right \vert ) .
\]
As announced, such a result is true in a more general situation, as suggested by the following result. Denote $\D (\cE _\mu)$ the domain of the closure of the pre-Dirichlet form $\cE _\mu$ defined initially on the space of $\C^\infty $ real-valued functions $f$ on $[\alpha,\beta]$ with vanishing derivative at the boundary, and let $\C _+ ^\infty ([\alpha, \beta])$ be the set of positive smooth functions on the interval $[\alpha, \beta]$.
\begin{theo}
\label{theo:neumann}
Assume $\AA$ (restricted to $[\alpha, \beta]$) and that $V_a$ is bounded from below on $[\alpha,\beta]$ for some function $a \in \C _+ ^\infty ([\alpha, \beta])$. Letting $f\in \D (\cE _\mu )$, then the following intertwining relation holds:
\[
\nabla _a P_t f = \tilde{P} _{a,t} ^{V_a} \nabla _a f , \quad t \geq 0,
\]
where $(P_t)_{t\geq 0}$ and $(\tilde{P}_{a,t})_{t\geq 0}$ denote the Neumann and Dirichlet semigroups associated to the processes $X$ and $X^a$, respectively. In particular we have the inequality
\[
\left \vert \nabla _a P_t f \right \vert \leq P_{a,t} ^{V_a} \left( \vert \nabla_a f \vert \right), \quad t \geq 0,
\]
where $(P_{a,t})_{t\geq 0}$ stands for the Neumann semigroup associated to $X^a$.
\end{theo}
\begin{proof}
The proof is similar to the one provided for Theorem~\ref{theo:intert_a}. Indeed, since the Neumann semigroup $(P_t)_{t\geq 0}$ satisfies $\nabla _a P_t f (\alpha) = \nabla _a P_t f (\beta) = 0$, we deduce that the same function $J$ defined by $J(t) := \nabla_a P_t f$ is a solution to the Schr\"odinger equation
\[
\left \{
\begin{array}{lll}
\partial_t u & = & \L_a ^{V_a} u \\
u(\cdot ,0) & = & \nabla_a f
\end{array}
\right.
\]
with the additional Dirichlet boundary conditions $u(\alpha, t) = u(\beta, t) = 0$. Then the uniqueness of the solution to this equation (as a consequence of the maximum principle), achieves the proof.
\end{proof}

Another possible extension of the intertwining method emphasized in Theorem~\ref{theo:intert_a} might be performed with respect to the dimension. Indeed, one would expect in this case a generator in the right-hand-side in \eqref{eq:commut_a} acting on 1-forms and not on functions. However the transfer at the level of the semigroups is not so clear. Some work have been done in this direction for Brownian motion on a Riemannian manifold through the so-called Weitzenb\"ock formula, which involves the Laplace-Beltrami operator (the generator of the Brownian motion), the Hodge Laplacian (an operator commuting with the gradient) and the Ricci curvature transform (the potential, or zero order operator) \cite{hsu, qian}. See also the litterature on semi-classical analysis of the Witten Laplacian \cite{helffer}, which has been introduced by Witten \cite{witten} by distorting the Hodge Laplacian with a Morse function. \vspace{0.1cm}

Before turning to our second main result, let us recall the following $L^\infty$ parabolic comparison principle, for instance on the basis of \cite{baudoin-garofalo}. Remember that for a given $a \in \C _+ ^\infty (\R)$ the potential $V_a$ is defined by
\[
V_a := \frac{\L _a (a)}{a} - b' .
\]
\begin{lemme}
\label{lemme:parabolic}
Assume $\AA$ and that $V_a$ is bounded from below for some $a\in \C _+ ^\infty (\R)$. Assume moreover that the process $X^a$ is non-explosive. Given a finite time horizon $t>0$, let $u = u_s(x)$ be a smooth bounded function on $\R \times [0,t]$. If the inequality
\[
\partial_t u + \L_a ^{2V_a} u  \geq 0
\]
holds on $\R \times [0,t]$, then we have
\[
P_{a,t} ^{2V_a} u_t \geq u _0 .
\]
\end{lemme}
\begin{proof}
Since the process is non-explosive, the sequence of stopping times
$$
\tau_n := \inf \{ t\geq 0 : \vert X^a _t \vert \geq n \} , \quad n \in \N ,
$$
goes to infinity almost surely as $n$ tends to infinity. By It\^o's formula and our assumption, we have for every $t\geq 0$,
\begin{eqnarray*}
e^{ - \int_0 ^t 2 V_a(X_s ^a) \, ds} \, u( X^a _t , t) & = & u (X^a _0 , 0) + M_{t} + \int _0 ^{t} \left( \L _a ^{2 V_a} + \partial _s u \right) (X^a _s , s) \, ds \\
& \geq & u (X^a _0 , 0) + M_{t},
\end{eqnarray*}
where $M$ is a local martingale. Hence the stopped process $(M_{t\wedge \tau_n})_{t\geq 0}$ is a true martingale and taking expectation, we get
\begin{eqnarray*}
\E _x \left[ e^{ - \int_0 ^{t\wedge \tau_n} 2 V_a(X_s ^a) \, ds} \, u( X^a _{t\wedge \tau_n} , {t\wedge \tau_n}) \right] & \geq  & u (x, 0) .
\end{eqnarray*}
Since $u$ is bounded and $V_a$ is bounded from below the dominated convergence theorem entails as $n\to +\infty$ the inequality
\begin{eqnarray*}
\E _x \left[ e^{ - \int_0 ^{t} 2 V_a(X_s ^a) \, ds} \, u( X^a _{t} , {t}) \right] & \geq  & u (x, 0) ,
\end{eqnarray*}
from which the desired result follows.
\end{proof}
Now we can state our bivariate convex version of Theorem~\ref{theo:intert_a}, which will be useful when dealing with other functional inequalities than Poincar\'e inequality, for which Theorem~\ref{theo:intert_a} will be sufficient. Let $\I$ be an open interval of $\R$ and denote $\mathcal{C}_\I$ the set of smooth convex functions $\varphi:\I \to\R$ such that $\varphi '' >0$, $\varphi '''$ is of constant sign and $-1/\varphi ''$ is convex on $\I $. For a given function $\vphi \in \mathcal{C}_\I$, we define the non-negative bivariate function $\Theta $ by
\begin{equation}
\label{eq:theta}
\Theta (x,y) := \vphi ''(x) \, y ^2, \quad (x,y) \in \I \times \R .
\end{equation}
By Theorem~4.4 in \cite{chafai2}, $\Theta$ is convex on $\I \times \R$. Some interesting examples of such functionals will be given in the next part. Since $\I$ may not include 0, define the set $\C ^\infty _0 (\R, \I)$ of functions $f : \R \to \I$ such that $f' \in \C ^\infty _0 (\R)$, which will play the role of smooth and compactly supported functions.
\begin{theo}
\label{theo:bicommutation} Assume $\AA$ and that $V_a$ is bounded from below for some $a\in \C _+ ^\infty (\R)$. Assume moreover that the processes $X$ and $X^a$ are non-explosive. Let $f \in \C ^\infty _0 (\R, \I)$ be such that for every $t\geq 0$,
\begin{equation}
\label{eq:assump}
\left( \frac{\sigma}{a} \right) ' \, \vphi ''' (P_t f)  \, (P_t f)'  \geq 0.
\end{equation}
Then we have the sub-intertwining inequality
\begin{equation}
\label{eq:bicommutation}
\Theta \left( P_t f , \nabla _a P_t f \right) \leq P_{a,t} ^{2V_a} \Theta (f,\nabla _a f) , \quad t\geq 0 .
  \end{equation}
\end{theo}
\begin{proof}
The proof is somewhat similar to that of Theorem~\ref{theo:intert_a}, except that it requires the $L^\infty$ parabolic comparison principle of Lemma~\ref{lemme:parabolic} because of the additional ingredient of convexity. For every $t>0$, define the function
\[
s\in[0,t]\mapsto J (s) :=  \Theta (P_{t-s} f,\nabla _a P_{t-s} f) .
\]
Since $f \in \C ^\infty _0 (\R, \I)$, $f$ is valued in a compact interval $[m,M] \subset \I$ and belongs to the space $\Lip (d_a)$. Moreover $X$ is non-explosive and thus from the identity $P_t 1 = 1$ one deduces that $P_t f$ is also valued in $[m,M]$, inducing the boundedness of the function $\phi '' (P_t f)$. Hence by Theorem~\ref{theo:intert_a} the function $J$ is bounded on $\R \times [0,t]$. Therefore the desired conclusion will hold once we have established that $J$ satisfies the inequality of Lemma~\ref{lemme:parabolic}. \vspace{0.1cm}

Using the intertwining relation \eqref{eq:commut_a}, we have
\begin{eqnarray*}
\L_a^{2 V_a} J(s) + \partial_s J(s) & = & \L_a^{2 V_a} \Theta (P_{t-s} f,\nabla _a P_{t-s} f) - \partial _x \Theta (P_{t-s} f,\nabla_a P_{t-s} f) \, \L P_{t-s} f \\
& & - \partial _y \Theta (P_{t-s} f,\nabla_a P_{t-s} f) \, \L _a ^{V_a} \nabla_a P_{t-s} f.
\end{eqnarray*}
Since $\Theta$ is a bivariate convex function we have
$$
\L _a \Theta (F,G) \geq \partial _x \Theta (F,G) \, \L _a F + \partial _y \Theta (F,G) \, \L _a G ,
$$
for every smooth functions $F,G$ and thus we finally obtain
\begin{eqnarray*}
\L_a^{2 V_a} J(s) + \partial_s J(s) & \geq & \partial _x \Theta (P_{t-s} f,\nabla_a P_{t-s} f) \, \left( \L _a P_{t-s} f - \L P_{t-s} f \right) \\
    & & + V_a \, \left( \partial _y \Theta (P_{t-s} f,\nabla_a P_{t-s} f) \, \nabla_a P_{t-s} f - 2 \, \Theta (P_{t-s} f,\nabla_a P_{t-s} f)\right) \\
    & = & 2 \, \sigma \, \left( \frac{\sigma}{a} \right) ' \, \vphi ''' (P_{t-s} f) \, (\nabla _a P_{t-s} f) ^3 \\
    & \geq & 0.
  \end{eqnarray*}
The proof is complete.
\end{proof}
Let us comment the previous result. As expected, the case $(\sigma /a )'= 0$ is related to the Bakry-\'Emery criterion, cf. the discussion above. In particular no assumption on the monotonicity of $f$ is required, as in the case when $\vphi$ is polynomial of degree 2, for which Theorem~\ref{theo:bicommutation} is a straightforward consequence of Theorem~\ref{theo:intert_a} and Jensen's inequality. Actually, the interesting cases are the ones for which \eqref{eq:assump} requires some restrictions on the functions $f, \sigma$ and $a$. For instance the convex functions $\vphi \in \C _\I$ we have in mind for the applications in Section~\ref{sect:appli} are $\vphi (x) := x \log x$ or $\vphi (x) = x^p$, $p\in (1,2)$ with both $\I = (0 ,\infty)$. Such functionals have negative third derivative on $\I$ and thus \eqref{eq:assump} means that we have to compare the monotonicity of $f$ and $\sigma /a$ since by Theorem~\ref{theo:intert_a}, $f$ and $P_t f$ are comonotonic functions.

\section{Application to functional inequalities}
\label{sect:appli} \textrm{} \indent
In this part we apply our main results Theorems~\ref{theo:intert_a} and \ref{theo:bicommutation} to functional inequalities. In particular the intertwining approach allows us on the one hand to recover the famous variational formula of Chen and Wang on the spectral gap \cite{chen_wang2}, and on the other hand to establish a restricted version of $\vphi$-entropy inequalities like logarithmic Sobolev or Beckner inequalities. \vspace{0.1cm}

Assume that the measure $\mu$ is a probability measure. Letting $\vphi \in \mathcal{C}_\I$, we define the $\varphi$-entropy of a function $f:\R \to\I$ such that $\vphi (f) \in L^1 (\mu)$ as
\[
\Ent_\mu^\vphi (f) := \mu \left( \vphi (f) \right) - \vphi \left( \mu (f) \right),
\]
where $\mu (g)$ stands for the integral of $g$ with respect to $\mu$. Note that the functional $\Ent_\mu^\vphi (f)$ is well-defined and non-negative by convexity of the function $\vphi$. Denote $H^1 _\vphi$ the set of $\I$-valued functions $f\in \D(\cE_\mu)$ such that $\vphi '(f) \in \D(\cE_\mu)$. We say that the $\varphi$-entropy inequality is satisfied with constant $c>0$ if for every $f\in H^1 _\vphi$,
\begin{equation}
\label{eq:phi_entropy}
c\, \Ent_\mu^\vphi (f) \leq \cE_\mu \left( f, \vphi '(f) \right).
\end{equation}
See for instance \cite{chafai2} for a careful study of the properties of $\varphi$-entropies. In particular the previous inequality can be rewritten as
\[
c\, \Ent_\mu^\vphi (f) \leq \int _\R \Theta ( f, \nabla_\sigma f) \, d\mu ,
\]
where $\Theta$ is the bivariate function defined in \eqref{eq:theta}. The $\varphi$-entropy inequality \eqref{eq:phi_entropy} is satisfied if and only if the following dissipation of the semigroup holds: for every $\I$-valued function $f$ such that $\vphi (f) \in L^1 (\mu)$ and every $t\geq0$,
\[
\Ent_\mu^\vphi (P_t f) \leq e^{-ct} \, \Ent_\mu^\vphi (f).
\]
As announced, below are listed some basic examples of $\vphi$-entropy inequalities. First we obtain the Poincar\'e inequality when $\vphi (x) := x^2$ with $\I = \R$:
\[
\CP \, \Var_\mu (f) \leq 2\, \cE _\mu (f,f),
\]
where $\Var_\mu (f) := \mu (f^2)-\mu(f)^2$ is the variance of $f$ under $\mu$. The optimal (largest) constant $\lambda_1 = \CP /2$ is the spectral gap in $L^2(\mu)$ of the operator $- \L$, i.e.
\begin{equation}
\label{eq:lambda1}
\lambda_1 = \inf _{f\in \D (\cE _\mu)} \frac{\cE _\mu (f,f)}{\Var _\mu (f)}.
\end{equation}
The spectral gap governs the $L^2(\mu)$ exponential decay to the equilibrium of the semigroup. On the other hand when $\vphi (x) := x \log x$ with $\I = (0,\infty)$ we obtain the logarithmic Sobolev inequality (or log-Sobolev inequality)
\[
\CLS \, \Ent_\mu (f) \leq \cE _\mu (f,\log f),
\]
where $\Ent_\mu (f) := \mu (f\, \log f)-\mu(f) \, \log \mu (f) $ is the entropy of $f$ under $\mu$. Such a functional inequality, which was originally introduced by Gross \cite{gross} to study hypercontractivity properties, is related to the entropy dissipation of the semigroup and is stronger than the Poincar\'e inequality (we have $\CLS \leq \CP$). Finally the third example we have in mind is the Beckner inequality which is obtained when considering the function $\varphi(x) = x^p $ with $p\in (1,2)$ and $I=(0,\infty)$. We have in this case
\[
\CB \, \left( \mu (f^p) - \mu(f) ^p \right) \leq p \, \cE _\mu \left( f,f^{p-1}\right) .
\]
Estimating the best constant in this inequality gives the decay in $L^p(\mu)$. Such an inequality was introduced by Beckner \cite{beckner} for the Gaussian measure under an alternative, but equivalent, formulation. Moreover it interpolates between Poincar\'e and log-Sobolev since it reduces to Poincar\'e if $p \to 2$, whereas we obtain log-Sobolev when dividing both sides by $p-1$ and taking the limit as $p$ goes to 1. \vspace{0.1cm}

As we have seen in Theorem~\ref{theo:bicommutation} above, the constant sign of the function $\vphi '''$ is of crucial importance. In particular this is the case for the three previous examples of functions $\vphi$. However there exist convex functions $\vphi$ satisfying all the assumptions provided in the very definition of $\C _\I$ except this point. An example is the opposite of the Gaussian isoperimetric function, that is,
\[
\varphi := - F' \circ F^{-1}  : (0,1) \to \R ,
\]
where $F$ is the Gaussian cumulative function $ F(t):=\int_{-\infty}^{t} e^{-x^2 /2}  \, dx/\sqrt{2\pi}$. Using the well-known relation $\varphi \, \varphi ''=-1$, it is straightforward to see that $\varphi '' >0$, that $-1/\vphi ''$ is convex and that $\varphi ''' (1/2) = 0$ with $\varphi '''$ negative on $(0,1/2)$ and positive on $(1/2,1)$. \vspace{0.1cm}

Let us start by the Poincar\'e inequality. Several years ago, Chen and Wang \cite{chen_wang2} used a coupling technique to establish a convenient variational formula on the spectral gap. In particular, the important point is that it provides in general ``easy-to-verify" conditions ensuring the existence of a spectral gap for the dynamics, together with qualitative estimates. Our next result allows us to recover simply this formula by using Theorem~\ref{theo:intert_a}. Recall that the potential $V_a$ is defined by
$$
V_a := \frac{\L _a (a)}{a} - b' ,
$$
for some function $a \in \C ^\infty _+ (\R)$, and define
\[
\rho_a := \inf_{x\in \R} V_a (x) ,
\]
when the infimum exists. Recall that if $\rho_a >0$ then the non-explosive process $X$ is positive recurrent and thus $\mu$ is normalized to be a probability measure. Note that $\sup_{a \in \C ^\infty _+ (\R)} \, \rho _a$ is always non-negative since $V_a$ is identically 0 when choosing $a := e^{-U}$, where $U$ is given in \eqref{eq:U}.
\begin{theo}[Chen-Wang~\cite{chen_wang2}]
\label{theo:spectral}
Assume that there exists some function $a \in \C _+ ^\infty (\R)$ such that $\rho_a >0$. Then the operator $-\L$ admits a spectral gap $\lambda_1$. More precisely the following formula holds:
\begin{equation}
\label{eq:chen_wang}
\lambda_1 \geq \sup_{a \in \C _+ ^\infty (\R)} \, \rho_a .
\end{equation}
In particular the equality holds if $\lambda_1$ is an eigenvalue of $-\L$.
\end{theo}
\begin{proof}
Letting $f\in \D(\cE_\mu)$ we have
\begin{eqnarray*}
\Var_\mu(f) & = & -\int_\R \int_0^{+\infty} \partial_t (P_t f)^2 \, dt \, d\mu \\
            & = & -2 \, \int_0^{+\infty} \int_\R P_t f \, \L P_t f \, d\mu \, dt \\
            & = & 2 \, \int_0^{+\infty} \int_\R \left( \nabla_\sigma P_t f \right) ^2 \, d\mu \, dt \\
            & = & 2 \, \int_0^{+\infty} \int_\R \left( P_{a,t} ^{V_a} \nabla _a f \right) ^2 \, d\mu _a \, dt \\
            & \leq & 2 \, \int_0^{+\infty} e^{-2\rho_a t} \int_\R P_{a,t} \left( (\nabla _a f)^2 \right) \, d\mu_a \, dt \\
            & \leq & 2 \, \int_0^{+\infty} e^{-2\rho_a t} \, dt \, \int_\R (\nabla_a f)^2 \, d\mu_a \\
            & = & \frac{1}{\rho_a} \, \int_\R (\nabla_\sigma f)^2 \, d\mu .
\end{eqnarray*}
To obtain above the lines 3, 4, 5 and 6, we used the integration by parts formula, the intertwining relation of Theorem~\ref{theo:intert_a}, Cauchy-Schwarz's inequality and the contraction property in $L^1 (\mu_a)$ of the semigroup $(P_{a,t})_{t\geq 0}$, respectively. Therefore we get $\lambda_1 \geq \rho_a$ from which we obtain the desired inequality \eqref{eq:chen_wang}. \vspace{0.1cm}

Now let us prove that the equality holds in \eqref{eq:chen_wang} when $\lambda_1$ is an eigenvalue of $-\L$, i.e. the equation
\begin{equation}
\label{eq:eigen}
-\L g = \lambda_1 \, g ,
\end{equation}
admits a non-constant smooth solution $g \in L^2 (\mu)$. The key point is to choose conveniently the function $a$ with respect to the eigenvector $g$. By \cite{chen_wang2} we already know that $g' >0$ (or $g' <0$) on $\R$. To see that the supremum is attained in \eqref{eq:chen_wang}, we differentiate on both sides of \eqref{eq:eigen} and use the intertwining relation \eqref{eq:commut_a} at the level of the generators:
\[
\lambda _1 \, \nabla _a g = - \nabla _a \L g = - \L _a ^{V_a} (\nabla_a g) = - \L _a \nabla_a g + V_a \, \nabla _a g.
\]
In the equalities above $g$ is chosen such that $g' >0$. Choosing the function $a = 1/g'$ entails $\lambda _1 = V_a$ identically. The proof of \eqref{eq:chen_wang} is now complete.
\end{proof}

We mention that the monotonicity of the eigenvector $g$ associated to $\lambda_1$ might be obtained directly by Theorem~\ref{theo:intert_a}. Given a function $f\in \D (\cE _{\mu})$, denote $v_f \in \D (\cE _{\mu})$ the function corresponding to the total variation of $f$, i.e. $v_f$ is absolutely continuous with weak derivative $\vert f'\vert$. Then we have $\cE _{\mu} (v_f, v_f) = \cE _{\mu} (f,f)$ and we obtain from Theorem~\ref{theo:intert_a} applied with $a = \sigma$ and Jensen's inequality:
\begin{eqnarray*}
\vert \nabla _\sigma P_t f (x) \vert & \leq & \E _x \left[ \vert \nabla _\sigma f(X_t) \vert \, \exp \left( - \int_0 ^t V_\sigma (X_s)\, ds \right) \right] \\
& = & \E _x \left[ \vert \nabla _\sigma v_f (X_t) \vert \, \exp \left( - \int_0 ^t V_\sigma (X_s)\, ds \right) \right] \\
& = & \vert \nabla _\sigma P_t v_f (x) \vert ,
\end{eqnarray*}
since $v_f$ is increasing. Hence we get in terms of variance,
\begin{eqnarray*}
\Var_\mu(f) & = & 2 \, \int_0^{+\infty} \int_\R  \vert \nabla _\sigma P_t f \vert ^2 \, d\mu \, dt \\
            & \leq & 2 \, \int_0^{+\infty} \int_\R  \vert \nabla _\sigma P_t v_f \vert ^2 \, d\mu \, dt \\
            & = & \Var_\mu (v_f).
\end{eqnarray*}
Since the analysis above is also available for the function $- v_f$ which is decreasing, one deduces that the definition \eqref{eq:lambda1} of the spectral gap is not altered if the infimum is taken over monotone functions $f\in \D (\cE _\mu)$. \vspace{0.1cm}

The proof of Theorem~\ref{theo:spectral} being based on Theorem~\ref{theo:intert_a}, whose analogue in the Neumann case is given by Theorem~\ref{theo:neumann}, the inequality \eqref{eq:chen_wang} of Theorem~\ref{theo:spectral} is also available for Neumann diffusions. Let us provide an alternative proof by means of the Sturm-Liouville comparison principle. Let $g$ be the first non-trivial (i.e. non-constant) eigenvector of the Neumann operator, i.e.
\[
- \L g =  \lambda_1 \, g, \quad g'(\alpha) = g'(\beta) = 0.
\]
Taking derivative and setting $G=g'$ give
\[
\sigma^2 \, G'' + (b+2\sigma \sigma') \, G' +(b'+\lambda_1)\, G = 0, \quad G(\alpha) = G(\beta) = 0.
\]
Let $u$ be a smooth function on $[\alpha,\beta]$ such that $u'>0$ and
let $U=u'$. If we choose $a= 1/U$ then we have
\[
V_a = - \frac{\left(\L u \right)'}{u'}.
\]
The potential $V_a$ is bounded from below by some constant $c_a >0$ if and only if
\[
\sigma^2 \, U'' + (b+2\sigma \sigma') \, U' +(b'+c_a) \, U \leq 0 .
\]
In other words there exists some smooth function $\nu : [\alpha,\beta] \to [0,+\infty)$ such that
\[
\sigma^2 \, U'' + (b+2\sigma \sigma') \, U' +(b'+c_a +\nu) \, U= 0.
\]
Assume that $V_a \geq c_a $ for some constant $c_a>0$ and that $\lambda_1 < c_a$. Since $b' +c_a +\nu > b' +\lambda_1$ the famous Sturm-Liouville comparison principle tells us that between (strictly) two zeros of $G$ there is a zero of $U$. Therefore we obtain a contradiction because we have $G(\alpha) = G(\beta) =0$ on the one hand and $U>0$ on $(\alpha,\beta)$ on the other hand. Hence we get $\lambda_1 \geq c_a$ and optimizing on the set $a\in \C _+ ^\infty (\R)$ gives the inequality \eqref{eq:chen_wang}. \vspace{0.1cm}

Actually, a refinement of Theorem~\ref{theo:spectral} might be obtained by using more carefully the properties of the Feynman-Kac semigroup. The following result is a kind of Brascamp-Lieb inequality, cf \cite{brascamp-lieb}. Given the (positive) Feynman-Kac potential $V_a$, denote
\begin{equation}
\label{eq:Lumer}
\Lambda _a (V_a) := \inf \left \{ - \int_\R g \, \L_a ^{V_a} g \, d\mu_a : g\in \D (\cE _{\mu_a} ^{V_a}); \, \Vert g \Vert _{L^2 (\mu_a)} = 1 \right \} ,
\end{equation}
where $\D (\cE _{\mu_a} ^{V_a}) := \D (\cE _{\mu_a}) \cap L^2 (V_a \, d\mu_a )$.
\begin{theo}
\label{theo:brascamp-lieb}
Assume that there exists some function $a \in \C _+ ^\infty (\R)$ such that $\rho_a >0$. Then the following Brascamp-Lieb type inequality holds: for every $f\in \D(\cE_\mu)$,
\[
\Var_\mu(f) \leq \int_\R \frac{\vert \nabla_\sigma f \vert ^2}{V_a} \, d\mu .
\]
\end{theo}
\begin{proof}
First note that the Schr\"odinger operator $\L _a ^{V_a}$ is invertible on the space $L^2 (\mu_a)$. Indeed for every $f\in L ^2 (\mu_a)$ we have by the Lumer-Phillips theorem,
\[
 \Vert P_{a,t} ^{V_a} f \Vert _{L^2 (\mu_a)} \leq e^{-t \Lambda _a (V_a)} \, \Vert f \Vert _{L^2 (\mu_a)} .
\]
Since $-\L_a$ is a non-negative operator we have $\Lambda _a (V_a)  \geq \rho_a$ and thus we obtain for every $f\in L ^2 (\mu_a)$,
\[
\left \Vert ( - \L _a ^{V_a} )^{-1} f \right \Vert _{L^2 (\mu_a)} \leq \frac{1}{\rho_a} \, \Vert f \Vert _{L^2 (\mu_a)}.
\]
Now we have for every $f\in \D(\cE_\mu)$,
\begin{eqnarray*}
\Var_\mu(f) & = & \int_\R f \, \left( f - \mu (f) \right) \, d\mu \\
            & = & - \int_\R  \int_0^{+\infty}  f\, \L P_t f \, dt \, d\mu \\
            & = & \int_0^{+\infty} \int_\R \nabla_\sigma f \, \nabla _\sigma P_t f \, d\mu \, dt \\
            & = & \int_0^{+\infty} \int_\R \nabla_a f \, \nabla _a P_t f \, d\mu_a \, dt \\
            & = & \int_0^{+\infty} \int_\R \nabla_a f \, P_{a,t} ^{V_a} \nabla_a f \, d\mu_a \, dt \\
            & = & \int_\R \nabla_a f \, ( - \L_a ^{V_a} ) ^{-1} \nabla_a f \, d\mu_a \\
            & \leq & \int_\R \nabla_a f \, ( V_a ) ^{-1} \, \nabla_a f \, d\mu_a \\
            & = & \int_\R \frac{\vert \nabla_\sigma f \vert ^2}{V_a} \, d\mu .
\end{eqnarray*}
To obtain the lines 5 and 7 we used respectively Theorem~\ref{theo:intert_a} and the standard inequality $(- \L _a + V_a)^{-1} \leq (V_a)^{-1}$ understood in the sense of non-negative operators, the operator $(V_a)^{-1}$ being the mutiplication by the function $1/V_a$. The proof is now complete.
\end{proof}

A consequence of the previous result is the following: every Lipschitz function with respect to the metric $d_\sigma$ has its variance controlled by the $L^1$-norm under $\mu$ of the function $1/V_a $. The variance of Lipschitz functions reveals to be an important quantity arising in various problems. For instance it has been studied in \cite{alon} through the so-called spread constant, in relation with concentration properties and isoperimetry, and has been revisited in \cite{sammer} through a mass transportation approach. Recently and under the Bakry-\'Emery criterion $\Gamma _2 \geq 0$, Milman showed in \cite{milman} that it is enough to bound the $L^1$-norm of centered Lipschitz functions to get a Poincar\'e inequality (with a universal loss in the constants). \vspace{0.1cm}

The next theorem is an integrated version of the inequality $\lambda_1 \geq \rho_\sigma$ which derives from Theorem~\ref{theo:spectral} or from the Bakry-\'Emery criterion \eqref{eq:BE}. Besides the clear improvement given by this integrated criterion, it reveals to be relevant when $V_\sigma$ is positive but tends to 0 at infinity, as we will see later with some examples. For a (compact) Riemannian manifold version of Theorem~\ref{theo:veysseire} below, we mention the recent work of Veysseire \cite{veysseire} in which the potential $V_\sigma$ is nothing but the Ricci curvature lower bound. Similarly to \eqref{eq:Lumer} define for the (positive) Feynman-Kac potential $V_\sigma$,
\[
\Lambda (V_\sigma) := \inf \left \{ - \int_\R g \, \L ^{V_\sigma} g \, d\mu : g\in \D (\cE _{\mu} ^{V_\sigma}); \, \Vert g \Vert _{L^2 (\mu)} = 1 \right \} ,
\]
where $\D (\cE _{\mu} ^{V_\sigma}) := \D (\cE _{\mu}) \cap L^2 (V_\sigma \, d\mu )$. With the notation of \eqref{eq:Lumer} it corresponds to the quantity $\Lambda_\sigma (V_\sigma)$.
\begin{theo}
\label{theo:veysseire}
Assume that the Feynman-Kac potential $V_\sigma$ is positive. Then we have the estimate
\begin{equation}
\label{eq:lambda1_integre}
\lambda _1 \, \geq \, \frac{1}{\int_\R \frac{1}{V_\sigma} \, d\mu} .
\end{equation}
\end{theo}
\begin{proof}
If the function $1/V_\sigma$ is not integrable with respect to $\mu$ then there is nothing to prove, hence let us assume that $1/V_\sigma \in L^1 (\mu)$. We will use a localization procedure. Let $R>0$ be a truncation level and consider the Neumann diffusion in the compact interval $[-R,R]$. Recall that the reversible measure $\mu _R$ is the original one restricted to the interval $[-R,R]$. If $\lambda_1 ^R$ denotes the spectral gap associated to the Neumann dynamics then we have $\lambda_1 ^R \downarrow \lambda_1$ as $R\to +\infty$. Hence if we establish the inequality
\[
\lambda_1 ^R \, \geq \, \frac{1}{\int_\R \frac{1}{V_\sigma } \, d\mu _R},
\]
then passing through the limit we obtain the desired estimate \eqref{eq:lambda1_integre}. Note that the potential $V_\sigma$ remains the same as for our original diffusion on $\R$. Therefore, without loss of generality we can assume that our diffusion is a compactly supported Neumann diffusion. In the rest of the proof we remove the superscript $R$ to avoid a saturated notation. The important point in this localization resides in the following fact: the potential $V_\sigma$ is bounded from below on $[-R,R]$ by some positive constant, say $\rho_\sigma$, hence the Neumann diffusion admits a spectral gap. \vspace{0.1cm}

Let us show on the one hand that $\lambda_1 \geq \Lambda (V_\sigma)$. Letting $f\in \D (\cE _\mu)$ be non-null and centered and $t\geq 0$, we have by Theorems~\ref{theo:brascamp-lieb} and \ref{theo:neumann} in the Neumann case,
\begin{eqnarray*}
\Var _\mu (P_t f) & \leq & \frac{1}{\rho_\sigma} \, \int_{-R} ^R ( \nabla_\sigma P_t f ) ^2 \, d\mu \\
& \leq & \frac{1}{\rho_\sigma} \, \Vert P_{t} ^{V_\sigma} \vert \nabla _\sigma f \vert \Vert ^2 _{L^2 (\mu)} \\
& \leq & \frac{1}{\rho_\sigma} \, e^{-2 t \Lambda (V_\sigma)} \, \cE_\mu (f,f),
\end{eqnarray*}
where we used the Lumer-Phillips theorem to obtain the third inequality. Now consider the function $\psi : [0,+\infty) \to \R$ defined by
\[
\psi (t) := \log \int_{-R} ^R (P_t f)^2 \, d\mu ,
\]
so that the latter inequality rewrites as
\begin{equation}
\label{eq:psi_convex}
\psi (t) \leq C(f) - 2 t \Lambda (V_\sigma),
\end{equation}
where $C(f)$ is some positive constant depending on $f$. Differentiating two times the function $\psi$ with respect to the time parameter yields, after an integration by parts,
\[
\partial ^2 _{t} \psi (t) = \frac{4}{\left( \int_{-R} ^R (P_t f) ^2 \, d\mu \right) ^2} \, \left( \int_{-R} ^R (\L P_t f)^2 \, d\mu \, \int_{-R} ^R (P_t f)^2 \, d\mu - \left( \int_{-R} ^R \L P_t f \, P_t f \, d\mu \right) ^2 \right) ,
\]
a quantity which is non-negative by Cauchy-Schwarz' inequality. Hence the function $\psi$ is convex and thus from \eqref{eq:psi_convex} we obtain $\psi(t) \leq \psi (0)- 2 t \Lambda (V_\sigma)$, or in other words,
\[
\Var _{\mu}(P_t f)  \leq e^{-2  t \Lambda (V_\sigma) } \, \Var _\mu (f).
\]
By density of $\D (\cE _\mu)$ in $L^2 (\mu)$ the above estimate is available for every function $f\in L^2(\mu)$ and thus we get $\lambda_1 \geq \Lambda (V_\sigma)$ since the spectral gap $\lambda_1$ is the best constant such that the $L^2$ convergence above holds. \\
On the other hand we have by the Poincar\'e inequality,
\begin{eqnarray*}
\Lambda (V_\sigma) & \geq & \inf \left \{ \lambda_1 \, \left( 1 - \mu (f) ^2 \right) + \int_{-R} ^R V_\sigma \, f^2 \, d\mu : f \in \D (\cE _{\mu} ^{V_\sigma}); \, \Vert f \Vert _{L^2 (\mu)} = 1\right \} \\
& \geq & \lambda_1 + \inf \left \{ \int_{-R} ^R V_\sigma \, f^2 \, d\mu \, \left( 1 - \lambda_1 \, \int_{-R} ^R \frac{1}{V_\sigma} \, d\mu \right) : f \in \D (\cE _{\mu} ^{V_\sigma}); \, \Vert f \Vert _{L^2 (\mu)} = 1 \right \} ,
\end{eqnarray*}
where we used Cauchy-Schwarz' inequality. Combining with the preceding inequality $\lambda_1 \geq \Lambda (V_\sigma)$ entails that the last infimum above is non-positive. Now if the desired conclusion is false, i.e.
\[
1 - \lambda_1 \, \int_\R \frac{1}{V_\sigma} \, d\mu > 0,
\]
then this infimum is at least $\rho_\sigma$ which is positive on $[-R,R]$, leading thus to a contradiction. Therefore the inequality \eqref{eq:lambda1_integre} holds in the Neumann case. The proof is now achieved.
\end{proof}
In the spirit of Theorem~\ref{theo:brascamp-lieb} it is reasonable to wonder if Theorem~\ref{theo:veysseire} still holds with the function $\sigma$ replaced by some good function $a\in \C ^\infty _+ (\R)$. However the answer is negative when adapting the previous method since a process and its $h-$transform have the same spectral properties. If $h$ denotes the function $\sigma /a$ then a bit of analysis shows that we have the following equivalence:
\[
gh \in \D (\cE _{\mu} ^{V_\sigma}) \quad \mbox{and} \quad \Vert gh \Vert _{L^2 (\mu)} = 1 \quad \Longleftrightarrow \quad g \in \D (\cE _{\mu_a} ^{V_a}) \quad \mbox{and} \quad \Vert g \Vert _{L^2 (\mu_a)} = 1 .
\]
Therefore we obtain thanks to the $h$-transform identity \eqref{eq:h-transf}:
\begin{eqnarray*}
\Lambda (V_\sigma) & = & \inf \left \{ - \int_\R g h \, \L ^{V_\sigma} (g h) \, d\mu : gh \in \D (\cE _{\mu} ^{V_\sigma}); \, \Vert gh \Vert _{L^2 (\mu)} = 1 \right \} \\
& = & \inf \left \{ - \int_\R g \, \L ^{V_\sigma ^{(h)}} g \, d\mu _a  : g \in \D (\cE _{\mu_a} ^{V_a}); \, \Vert g \Vert _{L^2 (\mu_a)} = 1 \right \} \\
& = & \inf \left \{ - \int_\R g \, \L _a ^{V_a } g \, d\mu _a  : g \in \D (\cE _{\mu_a} ^{V_a}); \, \Vert g \Vert _{L^2 (\mu_a)} = 1 \right \} \\
& = & \Lambda _a (V_a) .
\end{eqnarray*}

Now we turn to the case of more general functions $\vphi \in \C _\I$. We establish below a $\vphi$-entropy inequality restricted to a class of functions, in the spirit of the modified log-Sobolev inequality emphasized by Bobkov and Ledoux \cite{bob_ledoux}. Recall that if for some $a\in \C ^\infty _+ (\R)$ we have $a \asymp \sigma$ then the processes $X$ and $X^a$ are of the same nature. If moreover $\rho_a >0$ then both are positive recurrent.
\begin{theo}
\label{theo:phi_entropy}
Assume that there exists some function $a\in \C ^\infty _+ (\R)$ such that $\rho_a >0$ and also $a \asymp \sigma$. Then the $\varphi$-entropy inequality \eqref{eq:phi_entropy} holds with constant $2\rho_a$, for every function $f\in H^1 _\vphi$ satisfying the assumption \eqref{eq:assump}.
\end{theo}
\begin{proof}
A density argument allows us to prove the result only for functions $f\in \C ^\infty _0 (\R, \I)$. We have by integration by parts,
\begin{eqnarray*}
\Ent_\mu^\vphi (f) & = & - \int_\R \int_0^\infty \partial _t \vphi (P_t f) \, dt \, d\mu \\
    & = & - \int_0^{+\infty} \int_\R \vphi '(P_t f) \, \L P_t f \, d\mu \, dt \\
    & = & \int_0^{+\infty} \int_\R \nabla _\sigma P_t f \, \nabla _\sigma \vphi '(P_t f) \, d\mu \, dt \\
    & = & \int_0^{+\infty} \int_\R \Theta \left( P_t f , \nabla _a P_t f \right) \, d\mu_a \, dt ,
\end{eqnarray*}
where we remind that the bivariate function $\Theta $ defined at the end of Section~\ref{sect:main} is given by
\[
\Theta (x,y) := \varphi ''(x) \, y^2 , \quad (x,y) \in \I \times \R .
\]
Using now Theorem~\ref{theo:bicommutation} and the contraction property in $L^1 (\mu_a)$ of the semigroup $(P_{a,t})_{t\geq 0}$, we obtain
\begin{eqnarray*}
\Ent_\mu^\vphi (f) & \leq & \int_0^{+\infty} e^{-2\rho_a t} \int_\R P_{a,t} \Theta \left( f , \nabla _a f \right) \, d\mu_a \, dt \\
    & \leq & \int_0^{+\infty} e^{-2\rho_a t} \, dt \, \int_\R \Theta \left( f , \nabla_a f \right) \, d\mu_a \\
    & = & \frac{1}{2\rho_a} \, \int_\R \nabla _\sigma f \, \nabla _\sigma \vphi '(f) \, d\mu \\
    & = & \frac{1}{2\rho_a} \, \cE_\mu \left( f, \vphi '(f) \right) ,
\end{eqnarray*}
which completes the proof.
\end{proof}
Let $\IC _+ ^\infty (\R) $ and $\DC _+ ^\infty (\R) $ be the subsets of $\C _+ ^\infty (\R) $ given by considering increasing and decreasing functions, respectively. In the case of the log-Sobolev or Beckner inequalities, i.e. $\vphi (x) := x \log x$ or $\vphi (x) = x^p$ respectively, both with $\I = (0,\infty)$, then $\vphi ''' <0$ and thus \eqref{eq:assump} is reduced to
\[
\left( \frac{\sigma }{a} \right)' \, f' \leq 0,
\]
since by Theorem~\ref{theo:intert_a}, the functions $f$ and $P_t f$ are comonotonic. In particular we obtain the desired functional inequality for functions in $\IC _+ ^\infty (\R) $ (resp. in $\DC _+ ^\infty (\R)$) if $\sigma /a \in \DC _+ ^\infty (\R)$ (resp. $\sigma /a \in \IC _+ ^\infty (\R)$). Note that for the log-Sobolev inequality, Miclo \cite{miclo} proved that one can restrict to monotone functions, that is, if the log-Sobolev inequality is satisfied for the class of monotone functions, then it holds actually for all functions and with the same constant. Thus combining with Theorem~\ref{theo:phi_entropy} we obtain the following result.
\begin{theo}
\label{theo:LS_miclo}
Assume that there exists two functions $a, \tilde{a} \in \C _+ ^\infty (\R)$ such that $\sigma /a \in \IC _+ ^\infty (\R)$, $\sigma /\tilde{a} \in \DC _+ ^\infty (\R)$, $\rho_a >0$, $\rho_{\tilde{a}} > 0$ and also $a \asymp \tilde{a} \asymp \sigma$. Then the log-Sobolev inequality holds. More precisely the following estimate on the log-Sobolev constant $\CLS$ holds:
\[
\CLS \, \geq \, 2 \, \min \left( \sup \left \{ \rho_a : \sigma /a \in \IC _+ ^\infty (\R), \, a \asymp \sigma \right \} , \sup \left \{ \rho_a : \sigma /a \in \DC _+ ^\infty (\R), \, a \asymp \sigma  \right \} \right).
\]
In particular if the probability measure $\mu$ is symmetric, i.e. its density is an even function, then the latter inequality reduces to
\[
\CLS \, \geq \, 2 \, \sup \left \{ \rho_a : \sigma /a \in \IC _+ ^\infty (\R) , \, a \asymp \sigma \right \} .
\]
\end{theo}
Unfortunately, such a result cannot be similarly stated for the Beckner inequality. Indeed, we ignore if the Beckner inequality restricted to the class of monotone functions is equivalent to the standard Beckner inequality.

\section{Examples}
This final part is devoted to illustrate the above functional inequalities by revisiting classical examples, for which new estimates on the optimal constants are derived. We first focus our attention on the case when the diffusion function $\sigma$ is constant, equal to 1. Then the Sturm-Liouville operator we consider is given by
$$
\L f := f'' - U' \, f' ,
$$
where $U$ is some smooth potential. Take $U(0)$ such that $e^{-U}$ is a density with respect to the Lebesgue measure. For the examples we have in mind (except the Gaussian case), the Bakry-\'Emery theory is not fully satisfactory since the Feynman-Kac potential $V_\sigma$, which rewrites since $\sigma$ is constant as
\[
V_\sigma := \frac{\L \sigma}{\sigma} + U'' = U'' ,
\]
is not bounded from below by some positive constant. In other words, the potential $U$ is not strictly convex and can even be concave in a localized region, as for the double-well example. \vspace{0.1cm}

Based on his work on Hardy's inequalities, let us start by recalling the famous result of Muckenhoupt \cite{mucken} which characterizes the dynamics satisfying the Poincar\'e inequality on $\R$. See also the paper of Miclo \cite{M} for an approach through the so-called path method. Denote the quantities
\[
B_m ^+ := \sup _{x\geq m} \int _x ^{+\infty} e^{-U(y)} \, dy \, \int_m ^x   e^{U(y)} \, dy \quad \mbox{ and } \quad B_m ^- := \sup _{x\leq m} \int _{-\infty} ^x  e^{-U(y)} \, dy \, \int_x ^m e^{U(y)} \, dy ,
\]
where $m$ is a median of the probability measure $\mu$ with density $e^{-U}$. Finally set $B_m := \max \{ B_m ^+ , B_m ^- \}$.
\begin{theo}[Muckenhoupt]
\label{theo:mucken}
The operator $- \L$ has a spectral gap $\lambda_1$ if and only if $B_m $ is finite. More precisely we have the inequalities
\[
\frac{1}{4 B_m} \leq \lambda_1 \leq \frac{2}{B_m }.
\]
\end{theo}
Although the quantity $B_m$ might be difficult to estimate, the important point is the following: every non-trivial upper bound on $B_m$ provides a lower bound on the spectral gap. Following this observation let us introduce the operator
\begin{equation}
\label{eq:dynamics_alpha}
\L f (x) = f''(x) - \vert x \vert ^{\alpha -1} \, \Sign (x) \, f' (x),
\end{equation}
corresponding to the potential $U (x) := |x|^\alpha /\alpha$, where $\alpha >0$ and $\Sign$ stands for the sign function on $\R$. Although the function $U$ might not be $\C ^2$ at the origin, it does not play an important role in our study and thus can be ignored, at the price of an unessential regularizing procedure. For $\alpha = 1$ the reversible probability measure is the (symmetrized) exponential exponential measure on $\R$ whereas for $\alpha = 2$ the underlying process is the Ornstein-Uhlenbeck process and $\mu$ is the standard Gaussian distribution. It is well-known that the operator $-\L$ admits a spectral gap if and only if $\alpha \geq 1$ and the log-Sobolev inequality is satisfied if and only if $\alpha \geq 2$, cf. for instance \cite{latala}. \vspace{0.1cm}

Recall the notation of the potential $V_a$ defined in \eqref{eq:Va},
\[
V_a := \frac{\L _a (a)}{a} - b' ,
\]
and also $\rho_a := \inf V_a$ when it exists. Starting with the Poincar\'e inequality, our objective is to find some nice function $a\in \C _+ ^\infty (\R)$ such that $\rho_a >0$. The case $\alpha = 2$ in \eqref{eq:dynamics_alpha} is well-known and we have $\lambda_1 = 1$. To recover this result through Theorem~\ref{theo:spectral}, choose $a = 1$ which gives $V_a = 1$ hence $\lambda_1 \geq 1$. Moreover by the proof of Theorem~\ref{theo:spectral}, we see that linear functions are extremal and thus $\lambda_1 = 1$. Certainly, this result is expected since the Bakry-\'Emery theory fits perfectly and gives the optimal results, or in other words, the choice $a = \sigma$ is optimal. \vspace{0.1cm}

Let us consider the case $\alpha = 1$ in \eqref{eq:dynamics_alpha}. Using the nice properties of the exponential distribution, a famous result of Bobkov and Ledoux \cite{bob_ledoux} states that $\lambda_1  = 1/4$. To recapture this result we proceed as follows. Set $a (x) := e^{- \vert x\vert /2}$. Of course $a$ is not smooth at the origin but it causes no trouble for the present example. We have $V_a = \delta _0 + 1/4$ where $\delta _0$ stands for the Dirac mass at point 0, so that we obtain $\lambda_1 \geq \rho_a = 1/4$. To get the reverse inequality, apply the Poincar\'e inequality to the sequence of functions $x \mapsto e^{ \alpha \vert x\vert}$ where $\alpha <  1/2$ and take the limit as $\alpha \uparrow 1/2$ which yields $\lambda_1 \leq 1/4$, and thus the desired equality. \vspace{0.1cm}

Now we focus on the case $\alpha \in (1,2)$  in \eqref{eq:dynamics_alpha}. Applying Theorem~\ref{theo:veysseire} to these dynamics entails the following lower bound,
\begin{equation}
\label{eq:lambda1_alpha}
\lambda_1 \geq \frac{\alpha -1} {\int_\R \vert x \vert ^{2-\alpha} \, \mu (dx)}  = (\alpha -1) \, \alpha ^{1 - 2/\alpha} \, \frac{\Gamma (1/\alpha)}{\Gamma ((3-\alpha)/\alpha)} ,
\end{equation}
where $\Gamma$ is the well-known Gamma function $\Gamma(u):=\int_0^{+\infty} x^{u-1}\, e^{-x} \, dx$, $u>0$. Such a result might be compared with that obtained from the Muckenhoupt criterion. More precisely since $\mu $ is symmetric then it has median 0 and we have $B_0 ^+ = B_0 ^- $. Therefore we get
\begin{eqnarray*}
B_0 ^+ & = & \sup _{x\geq 0} \int _x ^{+\infty} e^{- y^\alpha /\alpha} \, dy \, \int_0 ^x   e^{y^\alpha /\alpha} \, dy \\
& \leq & \alpha ^{2/\alpha} \, \Gamma (1+ 1/\alpha) ^2 ,
\end{eqnarray*}
where we used the trivial inequality $s^\alpha - r^\alpha \geq (s-r)^\alpha $ with $0 \leq r \leq s$, which is available since $\alpha >1$. Finally Muckenhoupt's criterion entails the estimate
\begin{equation}
\label{eq:mucken_alpha}
\lambda_1 \geq \frac{1}{4 \, \alpha ^{2/\alpha} \, \Gamma (1+ 1/\alpha) ^2} .
\end{equation}
One notices that our estimate \eqref{eq:lambda1_alpha} is worse as soon as $\alpha \approx 1$ (Muckenhoupt's estimate is sharp for $\alpha = 1$) but is better than \eqref{eq:mucken_alpha} otherwise (numerically for $\alpha $ at least 1.188). In order to obtain a convenient upper bound on $\lambda_1 $, we have to apply the Poincar\'e inequality with a suitable function. For instance let $f$ be the $\mu$-centered function $f(x):= \Sign (x) |x|^\varepsilon $, where $\varepsilon > 1/2$. Then using some symmetries and a change of variables, we have
\begin{eqnarray}
\label{eq:lambda_upper}
\nonumber \lambda_1 & \leq & \frac{\cE _{\mu } (f,f)}{\Var _{\mu } (f)} \\
& = & \varepsilon ^2 \, \alpha ^{-2/\alpha} \, \frac{\Gamma ((2\varepsilon -1)/\alpha )}{ \Gamma ((2\varepsilon +1)/\alpha )} ,
\end{eqnarray}
Choosing now $\varepsilon = 1$ shows that we have the upper bound
\[
\lambda_1 \leq \frac{1}{3-\alpha} \, \alpha ^{1 - 2/\alpha} \, \frac{\Gamma (1/\alpha)}{\Gamma ((3-\alpha)/\alpha)} ,
\]
which is nothing but the lower bound \eqref{eq:lambda1_alpha} times $1/(3-\alpha)(\alpha -1)$. \vspace{0.1cm}

Another example of interest is given by the case $\alpha = 4$ in \eqref{eq:dynamics_alpha}. Here the lack of strict convexity of the potential $U$ is located at the origin. Letting $a : =e^W$ where $W$ is a smooth function to be chosen later, we have
\begin{eqnarray*}
V_a & = & \frac{\L _a (a)}{a} + U ''  \\
& = & \frac{a''}{a} + b_a \, \frac{a'}{a} + U'' \\
& = & W'' + (W')^2 + b_a W' + U'' \\
& = & W'' - (W')^2 - U ' \, W' + U'' \\
& = & W'' - \left( W' + \frac{U'}{2} \right) ^2 + \frac{(U')^2}{4} + U'' \\
& = & Z' - Z^2 + \frac{U''}{2} + \frac{(U')^2}{4},
\end{eqnarray*}
with $Z := W' + U'/2$. One of the simplest choice is to take $Z (x) := \varepsilon x$ with $\varepsilon >0$ to be chosen below, that is,
\[
W (x) \, = \, - \frac{U(x)}{2} + \frac{\varepsilon x^2}{2}, \quad x\in \R .
\]
Obviously, one could choose for $Z$ a polynomial of higher degree but the optimization below would become much more delicate. Plugging then into the above expression entails
\begin{eqnarray*}
\rho_a & = & \varepsilon + \inf_{x\in \R} \, \left \{ \frac{(U')^2}{4} + \frac{U''}{2} - \varepsilon  ^2 x^2 \right \} \\
& = & \varepsilon + \inf_{x\in \R} \, \left \{ \frac{x^6}{4} + \left( \frac{3}{2} - \varepsilon ^2 \right) \, x^2 \right \} .
\end{eqnarray*}
Taking $\varepsilon := \sqrt{3/2}$ we get $\rho_a = \sqrt{3/2}$ and therefore we obtain $\lambda_1 \geq \sqrt{3/2} \approx 1.224$ which is better than Muckenhoupt's estimate \eqref{eq:mucken_alpha} which only yields $\lambda_1 \geq 1/8\Gamma (5/4) ^2 \approx 0.152$. Together with the upper bound \eqref{eq:lambda_upper}, also available for $\alpha =4$ and numerically minimal for $\varepsilon \approx 0.854$ with the value $\approx 1.426$, we obtain for this example $\lambda_1 \in [1.224, 1.426]$. \vspace{0.1cm}

The last example we have in mind is the case of the double-well potential. For instance let $U$ be given by
\begin{equation}
\label{eq:double-well}
U(x) := x^4 /4 - \beta x^2 /2 , \quad x\in \R ,
\end{equation}
with $\beta >0$. Such a potential is convex at infinity but exhibits a concave region near the origin, which increases as $\beta$ does. These dynamics satisfy the functional inequalities of interest (Poincar\'e and log-Sobolev) thanks to the strict convexity at infinity and using perturbation arguments or Wang's criterion on exponential integrability, see for instance \cite{wang0}. Using the same method as before and with the same choice of function $a = e^W$ with
\[
W (x) \, := \, - \frac{U(x)}{2} + \frac{\varepsilon x^2}{2} = - \frac{x^4}{8} + \left( \frac{\beta}{4} + \frac{\varepsilon}{2} \right) \, x^2 ,
\]
yields to the following estimate:
\begin{eqnarray*}
\rho_a & = & \varepsilon + \inf_{x\in \R} \, \left \{ \frac{(U')^2}{4} + \frac{U ''}{2} - \varepsilon  ^2 x^2 \right \} \\
& = & \, \varepsilon -\frac{\beta}{2}
+ \inf_{x\in \R} \, \left \{  \frac{x^2}{4} \left(x^2-\beta \right)^2+ \left( \frac{3}{2} - \varepsilon ^2 \right) \, x^2 \right \} .
\end{eqnarray*}
For instance if $0 < \beta<\sqrt 6$ then taking $\varepsilon :=\sqrt{3/2}$ the minimum of $V_a$ is attained in $0$ and thus $\rho_a = \sqrt{3/2}-\beta >0$. Otherwise the case $\beta \geq \sqrt 6$ requires tedious computations to find a good parameter $\varepsilon$ such that $\rho_a >0$, meaning that the concave region between the two wells is large. \vspace{0.1cm}

Now we turn to the case of log-Sobolev inequalities. Once again our goal is to find a nice test function $a\in \C^\infty _+ (\R)$ such that $\rho_a >0$ but under the additional monotonicity constraint of Theorem~\ref{theo:LS_miclo}. Recall that we have taken the diffusion function $\sigma$ to be constant and equal to 1 in all the examples of interest, hence the statement of Theorem~\ref{theo:LS_miclo} reduces to
\[
\CLS \, \geq \, 2 \, \sup \left \{ \rho_a : a \in \IC ^\infty _+ (\R), \, a \asymp 1 \right \} .
\]
Let us investigate our previous examples. As mentioned above, the dynamics \eqref{eq:dynamics_alpha} satisfies the log-Sobolev inequality if and only if $\alpha \geq 2$. For the Ornstein-Uhlenbeck potential corresponding to the case $\alpha = 2$, it is well-known that $\CLS = 2$. By taking $a\equiv1$ as for the Poincar\'e inequality we get $\CLS \geq 2$. Since $ 2 = 2\, \lambda_1 = \CP \geq \CLS$, we obtain $\CLS = 2$. We can also recover that the exponential functions $f_\kappa (x):= e^{\kappa x}$ are extremal for the log-Sobolev inequality. Indeed, using the famous commutation relation
\[
(P_t g)' = e^{-t} \, P_t (g'),
\]
available for every smooth positive function $g$, it can be shown that the equality holds in the inequality \eqref{eq:bicommutation} of Theorem~\ref{theo:bicommutation}, that is,
\[
\Theta \left( P_t f_\kappa , (P_t f_\kappa)' \right) = e^{-2t} \, P_t \Theta (f_\kappa , f_\kappa ') ,
\]
where in the log-Sobolev case $\Theta (r,s) := s^2 /r$ with $r>0$ and $s\in \R$. Then it is enough to observe that the equality is conserved in all the steps of the proof of Theorem~\ref{theo:phi_entropy}. \vspace{0.1cm}

We now  turn to the two other examples, i.e. the case of the generator \eqref{eq:dynamics_alpha} for $\alpha >2$ and the double-well potential defined in \eqref{eq:double-well}. As before we set $a:=e^W$ and since we require $a$ to be increasing, let $W' :=e^A$ where $A$ is some convenient function to be chosen below. As for the spectral gap, we have for every $x\in \R$,
\begin{eqnarray*}
V_a (x) & = & W''(x) - (W')^2 (x) - U' (x) \, W'(x) + U''(x) \\
& = & \left( A' (x)  -e^{A(x)} - \Sign (x) \, \vert x\vert ^{\alpha -1} \right) \, e^{A(x)}  + (\alpha -1) \, \vert x\vert ^{\alpha -2}.
\end{eqnarray*}
In order to obtain $\rho _a >0$ with furthermore $a \asymp 1$, we are looking for some function $A$ increasing in a neighborhood of 0, going to $-\infty$ as $x\to + \infty$ so that $e^A$ is integrable on $\R$ with respect to the Lebesgue measure, and which does not tend to $+\infty$ as $x\to -\infty$. Maybe far from optimality, an example is given by $A(x) := - (\varepsilon x - \gamma)^2$, where for each example of interest the non-null numbers $\varepsilon$ and $\gamma$ are chosen conveniently of the same sign. For instance in the case $\alpha = 4$ the choices $\varepsilon = 1$ and $\gamma = 1$ entail, using numerical computations, the lower bound $\rho_a \geq 0.594$. Together with the upper bound $\CLS \leq 2 \lambda_1$ we obtain $\CLS \in [1.188, 2.852]$. \vspace{0.1cm}

For the double-well example, the potential $V_a$ reads as
\[
V_a(x) \, = \, (A'(x) - e^{A(x)} -x^3 +\beta x) \, e^{A(x)} + 3x^2 -\beta.
\]
Once again the choice of $A(x) = - (\varepsilon x - \gamma)^2$, where $\varepsilon, \gamma$ are found such that $V_a >0$ in a neighborhood of 0, allows us to obtain $\rho_a >0$. For instance if $\beta = 1/2$ then set $\varepsilon = 1.28$ and $\gamma = 1$ so that we get $\rho_a \geq 0.22$ using again numerical computations and thus $\CLS \geq 0.44$. \vspace{0.1cm}

Let us achieve this work by considering an example involving a non-constant diffusion function $\sigma$. Given some potential $U$, we assume once again that the probability measure $\mu$ has density proportional to $e^{-U}$. Such a measure is reversible with respect to the following dynamics
\[
\L f \, = \, \sigma^2 \, f'' + \left( 2\, \sigma \, \sigma ' - \sigma ^2 \, U' \right) \, f ' ,
\]
where $\sigma$ is an arbitrary smooth function. We focus on the generalized Cauchy measure $\mu$ with density proportional to $(1 + x^2)^{-\beta}$, where $\beta > 1/2$. This means that the potential $U$ is given by
\[
U(x):= \beta \log (1+x^2) , \quad x\in \R .
\]
Denote $Z_\beta$ the normalization constant,
\[
Z_\beta := \int_\R \frac{dx}{(1 + x^2)^\beta } = \frac{\Gamma (1/2) \, \Gamma (\beta - 1/2)}{\Gamma (\beta)}.
\]
It is known that these dynamics do not satisfy the Poincar\'e inequality with the choice $\sigma = 1$ since the distance function $d_\sigma$ is not exponentially integrable. To overcome this difficulty we keep the same measure $\mu$ and choose conveniently the diffusion function $\sigma$. By a recent result of Bobkov and Ledoux \cite{bob_ledoux2}, the Poincar\'e inequality is satisfied with $\sigma (x) := \sqrt{1+x^2}$, i.e. for every $f\in \D (\cE _\mu)$,
\[
\frac{\beta-1}{2} \, \Var _\mu (f) \leq \int_\R ( 1+x^2) \, \left( f'(x) \right)^2 \, \mu (dx) .
\]
With this choice of diffusion function, brief computations give the potential
\begin{eqnarray*}
V_\sigma (x) & = & \frac{\L \sigma}{\sigma} - \left( 2\, \sigma \, \sigma ' - \sigma ^2 \, U' \right) ' \\
& = & -\sigma \sigma'' + \sigma \sigma' U' +\sigma^2 U'' \\
& = & \frac{2 \beta -1}{1+x^2}.
\end{eqnarray*}
Hence by Theorem~\ref{theo:veysseire} we obtain for every $\beta >3/2$ the following lower bound on the spectral gap $\lambda_1$:
\[
\lambda _1 \geq \frac{(2\beta -1) \, Z_\beta}{Z_{\beta - 1}} = \frac{(2\beta -1)(\beta - 3/2)}{\beta - 1},
\]
which can be compared to the constant above $(\beta-1)/2$. \vspace{0.1cm}

Moreover the log-Sobolev inequality holds with the different diffusion function $\tilde{\sigma} (x) := 1+x^2$, cf. \cite{bob_ledoux2}, and with constant $\CLS \geq 4(\beta -1)$ at least for $\beta >1$, that is, for every $f\in \D (\cE _\mu)$ such that $\log f \in \D (\cE _\mu)$,
\[
4 (\beta -1) \, \Ent_\mu (f) \leq \int_\R (1+x^2) ^2 \, f'(x) \, (\log f(x))' \, \mu (dx).
\]
The Bakry-\'Emery criterion, i.e. the choice $a=\tilde{\sigma}$ in Theorem~\ref{theo:LS_miclo} above, allows us to recover this result since we have
\[
V_{\tilde{\sigma}} ( x) = 2 \, (\beta -1) \,(1+ x^2 ), \quad x\in \R ,
\]
and taking the infimum yields $\rho_{\tilde{\sigma}} = 2 (\beta -1)$.

\section*{Acknowledgments}
The authors are grateful to D. Bakry for providing them a preliminary version of the forthcoming book \cite{bgl}. They also thank the french ANR projects Stab and GeMeCoD for financial support.

\end{document}